\documentclass[preprint,12pt]{elsarticle}

\usepackage{amssymb}
\usepackage{amscd}
\usepackage{amsthm}
\usepackage{amsmath}
\newtheorem{theorem}{Theorem}[section]
\newtheorem{definition}[theorem]{Definition}
\newtheorem{lemma}[theorem]{Lemma}

\newtheorem{proposition}[theorem]{Proposition}
\newproof{Proof}{Proof}
\numberwithin{equation}{section}
\theoremstyle{definition}

\begin{document}

\begin{frontmatter}



\title{Metric Reduction and Generalized Holomorphic Structures}


\author{Yicao Wang}

\address{Department of Mathematics, Hohai University, Nanjing 210098, China\\}
\ead{yicwang@hhu.edu.cn}

\begin{abstract}
In this paper, metric reduction in generalized geometry is investigated. We show how the Bismut connections on the quotient manifold are obtained from those on the original manifold. The result facilitates the analysis of generalized K$\ddot{a}$hler reduction, which motivates the concept of metric generalized principal bundles and our approach to construct a family of generalized holomorphic line bundles over $\mathbb{C}P^2$ equipped with some non-trivial generalized K$\ddot{a}$hler structures.
\end{abstract}


\end{frontmatter}


\section{Introduction}

Generalized complex geometry initiated by N. Hitchin and his school is a simultaneous generalization of symplectic geometry and complex geometry. Since Marsden-Weinstein reduction is a basic construction in symplectic geometry, it is natural to explore a generalized version of symplectic reduction in generalized geometry. This topic was treated in great generality in the formalism of Courant reduction in \cite{BCG1}. When furthermore there is a generalized metric on the Courant algebroid to be reduced, it also descends to the reduced Courant algebroid under proper conditions. In \cite{Ca}, this 'metric reduction' was investigated; in particular, this procedure was checked from the angle of geometry of tangent bundles. The present paper arises from our work \cite{Wang} on trying to understand metric reduction from a topological field theoretic viewpoint.

Considerations in generalized geometry are conceptually direct and useful, but the underlying structures often hide in depth and need careful analysis. For example, generalized K$\ddot{a}$hler reduction is easily understood from the general procedure of reduction of Dirac structures, but it contains some sophisticated details from the viewpoint of classical complex geometry. Some of these were included in \cite{Ca}. In this paper, we will carry on this investigation.

We pay much attention on the special case of \emph{isotropic trivially extended $G$-actions} in the sense of \cite{BCG1}, where $G$ is a compact connected Lie group. With an invariant generalized metric in place, the manifold $M$ under consideration carries two horizontal distributions $\tau_\pm$, which are central in our paper. Basically, they are used to express the Bismut connections in the reduced manifold $M_{red}:=M/G$ in terms of Bismut connections in $M$. This is different from the case of reducing the Levi-Civita connection on $M$--In the latter case, a connection of the principal bundle $M\rightarrow M_{red}$ naturally arises from the $G$-invariant metric $g$, i.e. the horizontal distribution is just the orthogonal complement $\mathcal{H}$ of the vertical distribution. The Levi-Civita connection on $M_{red}$ can then be expressed using the Levi-Civita connection on $M$ and the orthogonal projection from $TM$ to $\mathcal{H}$. As for reducing Bismut connections, it is not as directly solved as in the ordinary case and should be motivated by conceptual considerations in generalized geometry. This investigation of reducing Bismut connections is motivated by gauging a zero-dimensional supersymmetric $\sigma$-model in \cite{Wang}.

When the invariant generalized metric is from a generalized K$\ddot{a}$hler manifold $\mathcal{M}$, the situation becomes more interesting. To get a reduced generalized K$\ddot{a}$hler manfold, an invariant submanifold $M\subset \mathcal{M}$ should be carefully chosen and the reduced generalized K$\ddot{a}$hler structure will then sit on $M_{red}=M/G$. Hence $M$ only serves as an intermediate object in this procedure. But in this paper $M$ as a \emph{metric generalized principal bundle} (see \S \ref{mgp}) proves to have its own interest: The curvatures of $\tau_\pm$ are of type $(1,1)$ w.r.t. the reduced complex structures $\tilde{J}_\pm$ on $M_{red}$ respectively. Thus any associated complex vector bundle acquires simultaneously a $\tilde{J}_+$-holomorphic structure and a $\tilde{J}_-$-holomorphic structure.\footnote{Similar phenomenon, of course, occurs in ordinary K$\ddot{a}$hler reduction but is seldom emphasized in the literature.} This motivates our approach to constructing generalized holomorphic vector bundles from generalized K$\ddot{a}$hler reduction.

The paper is organized as follows. In \S~\ref{bas}, we review the basic content of generalized geometry. The goal of \S~\ref{sect3} is to lay the concrete background for later development by investigating the notion of isotropic trivially extended $G$-action in the presence of an invariant generalized metric. Compared with the work in \cite{Ca}, we hardly contain much essentially new content, but our viewpoint is slightly different. In particular, we include some details of the reduced structures which were missing in \cite{Ca}, and emphasize the basic role of the distributions $k_\pm$ (Eq.~(\ref{dec}) is essential for reducing the Bismut connections) which was not explicitly mentioned in \cite{Ca}. In \S~\ref{RBis}, we mainly tackle the problem of expressing the reduced Bismut connections in terms of Bismut connections in the original manifold (Thm.~\ref{Bism}). The curvature of the reduced Bismut connection is also computed in terms of the reduction data (Thm.~\ref{curvature}). These computations play a basic role in \cite{Wang}. The last three sections devote to using generalized K$\ddot{a}$hler reduction to produce generalized holomorphic vector bundles. \S~\ref{mgp} discusses the notion of metric generalized principal $G$-bundle and its associated relative curvature. \S~\ref{GKR} revisits generalized K$\ddot{a}$hler reduction in the spirit of previous sections, and emphasis is put on structures on the intermediate metric generalized principal $G$-bundle, which carries a biholomorphic structure. These two sections pave the way for us to produce generalized holomorphic vector bundles via generalized K$\ddot{a}$hler reduction in \S~\ref{ghs}. We give a sufficient condition for the biholomorphic structure to be generalized holomorphic in the Hamiltonian case. As examples, we have constructed generalized holomorphic line bundles on $\mathbb{C}P^2$ equipped with non-trivial generalized K$\ddot{a}$hler structures.
\section{Basics of generalized geometry}\label{bas}
In this section, we collect the most relevant aspects of generalized geometry. For a detailed account for it, we refer the reader to \cite{Gu00} \cite{Gu0}.

In generalized geometry, one considers geometric structures defined on the generalized tangent bundle $\mathbb{T}M=TM\oplus T^*M$ of a smooth manifold $M$, or more generally on an exact Courant algebroid over $M$.

A Courant algebroid $E$ is a real vector bundle $E$ over $M$, together with an anchor map $\pi$ to $TM$, a non-degenerate inner product and a so-called Courant bracket $[\cdot , \cdot]_c$ on $\Gamma(E)$. These structures should satisfy some compatibility axioms. $E$ is called exact, if the short sequence
\[0\longrightarrow T^*M\stackrel{\pi^*}\longrightarrow E \stackrel{\pi}\longrightarrow TM \longrightarrow0\]
is exact. In this paper, by 'Courant algebroid', we always mean an exact one. Given $E$, one can always find an isotropic right splitting $s:TM\rightarrow E$, which has a curvature form $H\in \Omega_{cl}^3(M)$ defined by
\[H(X,Y,Z)=\langle[s(X),s(Y)]_c,s(Z)\rangle,\quad X, Y, Z\in \Gamma(TM).\]
  By the bundle isomorphism $s+\pi^*:TM\oplus T^*M\rightarrow E$, the Courant algebroid structure can be transported onto $\mathbb{T}M$. Then the inner product $\langle\cdot,\cdot\rangle$ is the natural pairing, i.e.
$\langle X+\xi,Y+\eta\rangle=\xi(Y)+\eta(X)$, and the Courant bracket is
\begin{equation}[X+\xi, Y+\eta]_H=[X,Y]+\mathcal{L}_X\eta-\iota_Yd\xi+\iota_Y\iota_XH,\end{equation}
called the $H$-twisted Courant bracket. Different splittings are related by $B$-field transforms, i.e. $e^B(X+\xi)=X+\xi+\iota_XB$, where $B$ is a 2-form.

A maximal isotropic subbundle $L\subset E$ is called an almost Dirac structure. If $L$ is involutive w.r.t. the Courant bracket, it is called a Dirac structure. These notions can be extended directly to the complexified setting which interests us most.
\begin{definition}
 A generalized complex structure on $E$ is a complex structure $\mathbb{J}$ on $E$ orthogonal w.r.t. the inner product and whose $\sqrt{-1}$-eigenbundle $L\subset E\otimes\mathbb{C}$ is a complex Dirac structure.
\end{definition}
Since $\mathbb{J}$ and its $\sqrt{-1}$-eigenbundle $L$ are equivalent notions, we shall use them interchangeably to denote a generalized complex structure. At a point $x\in M$, the codimension of $\pi(L_x)$ in $T_xM\otimes \mathbb{C}$ is called the type of $\mathbb{J}$ at $x$. Type can vary along some subset of $M$, which makes the local geometry of generalized complex structures rather non-trivial.

A generalized complex structure $L$ is an example of complex Lie algebroids. Via the inner product, $\wedge^\cdot L^*$ can be identified with $\wedge^\cdot \bar{L}$, and we have an elliptic differential complex $(\Gamma(\wedge^\cdot \bar{L}), d_L)$, which induces the Lie algebroid cohomology associated with the Lie algebroid $L$. The differential complex can be twisted by an $L$-module.
\begin{definition}
Given a generalized complex structure $L$ over $M$, an
$L$-connection $\mathrm{D}$ in a complex vector bundle $W$ is a differential operator
$\mathrm{D}:\Gamma(W)\longrightarrow \Gamma(\bar{L}\otimes W)$
satisfying
$$\mathrm{D}(f s)=d_Lf\otimes s+f\mathrm{D}s,\quad s\in \Gamma(W),\ f\in C^\infty(M).$$
If $\mathrm{D}$ is flat, i.e. $\mathrm{D}^2=0$, $\mathrm{D}$ is called a generalized holomorphic structure and $W$ an $L$-module or a generalized holomorphic vector bundle.
\end{definition}
A standard non-trivial example of generalized holomorphic bundles is the canonical line bundle in the pure spinor description of a generalized complex structure. Further analysis and examples of generalized holomorphic structures can be found in \cite{Wang0} \cite{Wang1}.
\begin{definition}A generalized (Riemannian) metric on $E$ is an orthogonal, self-adjoint operator $\mathcal{G}$ such that $\langle \mathcal{G}e,e\rangle > 0$ for nonzero $e\in E$. It is necessary that $\mathcal{G}^2=id$. The $\pm$-eigenbundles $V_\pm$ are positive and negative subbudles of maximal rank respectively.
 \end{definition}
 A generalized metric induces a canonical isotropic splitting: $E=\mathcal{G}(T^*M)\oplus T^*M$. It is called the metric splitting. Given a generalized metric, we shall always choose the metric splitting to identify $E$ with $\mathbb{T}M$. Then $\mathcal{G}$ is of the form $\left(\begin{array}{cc} 0 & g^{-1} \\g & 0 \\
\end{array} \right)$ where $g$ is an ordinary Riemannian metric, and vectors in $V_\pm$ are of the form $X\pm g(X)$ respectively for $X\in TM$.

  If $H$ is the curvature of the metric splitting, sometimes we call the triple $(M, g, H)$ a generalized Riemannian manifold, without explicitly mentioning the underlying Courant algebroid and generalized metric. For a generalized Riemannian manifold $(M, g, H)$, one can define the Bismut connections $\nabla^\pm=\nabla\pm\frac{1}{2}g^{-1}H$, where $\nabla$ is the Levi-Civita connection. It was observed in \cite{Hi} \cite{Gu1} that these connections can be expressed using $H$-twisted Courant bracket:
\begin{equation}[X\mp g(X), Y\pm g(Y)]_H^\pm=\nabla_X^\pm Y\pm g(\nabla_X^\pm Y),\label{Hif}\end{equation}
where $(X+\xi)^\pm$ denote the $V_\pm$-part of $X+\xi\in \Gamma(\mathbb{T}M)$ w.r.t. the decomposition $E=V_+\oplus V_-$.

A generalized metric is an ingredient of a generalized K$\ddot{a}$hler structure, which is the analogue of K$\ddot{a}$hler structure in complex geometry.
\begin{definition}
A generalized K$\ddot{a}$hler structure on $E$ is a pair of commuting generalized complex structures $\mathbb{J}_1$ and $\mathbb{J}_2$ such that $\mathcal{G}=-\mathbb{J}_1 \mathbb{J}_2$ is a generalized metric.
\end{definition}

A generalized K$\ddot{a}$hler structure can also be characterized in terms of more ordinary notions: There are two complex structures $J_\pm$ on $M$ compatible with the Riemannian metric $g$ induced from the generalized metric. Let $\omega_\pm=gJ_\pm$ and $H$ be the curvature of the metric splitting. Then
\begin{equation}d_+^c \omega_+=-d_-^c\omega_-=-H,\end{equation}
where $d_\pm^c$ are the $d^c$-differentials associated to $J_\pm$ respectively. $J_\pm$ is necessarily flat w.r.t. $\nabla^\pm$ respectively and $H$ should be of type $(1,2)+(2,1)$ w.r.t. both $J_+$ and $J_-$. Let $T^\pm_{0,1}M$ be the anti-holomorphic tangent bundles w.r.t. $J_\pm$ respectively. Then we can form two vector bundles over $M$:
\[L_\pm=\{X\pm\sqrt{-1}\omega_\pm(X)|X\in T^\pm_{0,1}M \}.\]
In the metric splitting, $L_1:=L_+\oplus L_-$ and $L_2:=L_+\oplus \bar{L}_-$ are precisely $\sqrt{-1}$-eigenbundles of $\mathbb{J}_1$ and $\mathbb{J}_2$ respectively.

We are particularly interested in generalized holomorphic structures over a generalized K$\ddot{a}$hler manifold. In this setting, we choose $L_1$ to be the underlying generalized complex structure of a generalized holomorphic structure $\mathrm{D}$. Due to the decomposition $L_1=L_+\oplus L_-$, $\mathrm{D}$ can be decomposed as $\mathrm{D}=\bar{\delta}_++\bar{\delta}_-$ accordingly. Actually, $\bar{\delta}_\pm$ are ordinary $J_\pm$-holomorphic structures respectively. Additionally, it is necessary that \begin{equation}\bar{\delta}_+\bar{\delta}_-+\bar{\delta}_-\bar{\delta}_+=0.\label{ghc}\end{equation}
 Conversely, given $J_\pm$-holomorphic structures $\bar{\delta}_\pm$, if Eq.~(\ref{ghc}) is also satisfied, then $\mathrm{D}:=\bar{\delta}_++\bar{\delta}_-$ is a generalized holomorphic structure \cite{Hu}.

\section{Isotropic trivially extended action and metric reduction}\label{sect3}
Though there is a much more general framework in \cite{BCG1} to adapt an ordinary Lie algebra action to the setting of a Courant algebroid, we content ourselves here with the following more restrictive notion of isotropic trivially extended action of a Lie algebra $\mathfrak{g}$. Throughout the paper, we always assume that $\mathfrak{g}$ is the Lie algebra of a compact connected Lie group $G$ acting freely on $M$ from the left. In the following, a Courant algebroid $E$ over $M$ is fixed.

\begin{definition}\label{ext}\cite{BCG1}
Let $\varphi_0: \mathfrak{g}\rightarrow \Gamma(TM)$ be the infinitesimal action of $G$ over $M$. An isotropic trivial extension of this action to $E$ is a bracket-preserving morphism $\varphi:\mathfrak{g}\rightarrow \Gamma(E)$ such that the following diagram
\[\begin{CD}
\mathfrak{g} @> id >> \mathfrak{g} \\
@V\varphi VV      @ V\varphi_0 VV\\
\Gamma(E) @>>\pi> \Gamma(TM)
\end{CD}
\]
is  commutative and the image of $\varphi$ is isotropic pointwise in $E$. If furthermore this extended action integrates to a $G$-action on $E$, we call it an isotropic  trivially extended $G$-action. \end{definition}

Let $e_a$, $a=1,2,\cdots, \textup{dim} \mathfrak{g}$ be a basis of $\mathfrak{g}$ and let $V_a$ be the corresponding fundamental vector fields on $M$. When an isotropic splitting of $E$ is chosen, $\varphi(e_a)=V_a+\xi_a$, where $\xi_{(\cdot)}: \mathfrak{g}\rightarrow \Gamma(T^*M)$ is equivariant and $\iota_a \xi_b+\iota_b\xi_a=0$, where $\iota_a$ denotes contraction with $V_a$. If additionally the splitting is invariant, then its curvature $H$ is invariant and $\iota_a H=d\xi_a$. A remarkable fact in \cite{BCG1} is that $H+\xi_{(\cdot)}$ is actually a closed equivariant 3-form in the Cartan model of equivariant de Rham cohomology.

 Let $K\subset E$ be the subbundle generated by $\varphi(\mathfrak{g})$, and $K^\bot$ its orthogonal complement in $E$ w.r.t. the inner product. Then due to the reduction theory developed in \cite{BCG1}, $E_{red}:=\frac{K^\bot}{K}/G$ has the structure of a Courant algebroid induced from $E$. Now if $\mathcal{G}$ is a $G$-invariant generalized metric over $E$, then $E_{red}$ also acquires a generalized metric $\mathcal{G}_{red}$.

 There is a useful way to describe $\mathcal{G}_{red}$. Let $K^\mathcal{G}$ be the $\mathcal{G}$-orthogonal complement of $K$ in $K^\bot$, i.e. $$K^\mathcal{G}=\mathcal{G}(K^\bot)\cap K^\bot.$$ By projection, $K^\mathcal{G}$ is isomorphism to $K^\bot/K$, and $\mathcal{G}_{red}$ is actually the restriction of $\mathcal{G}$ on the subbundle $K^\mathcal{G}\subset K^\bot$. Accordingly, we have the decomposition
\[K^\mathcal{G}=V_+^{red}\oplus V_-^{red},\] where $V_\pm^{red}=V_\pm \cap K^\mathcal{G}$. Furthermore, by the abovementioned isomorphism, we can regard $K^\mathcal{G}/G$ as a Courant algebroid over $M_{red}$. The advantage of using $K^\mathcal{G}$ instead of $K^\bot/K$ is that, when a lift $\hat{A}\in \Gamma(K^\bot)$ of $A\in \Gamma(E_{red})$ is needed, we can choose $\hat{A}$ to be the unique one in $\Gamma(K^\mathcal{G})$.

Though the Courant algebroid structure of $K^\mathcal{G}/G$ is clear from the generalized reduction procedure, for later convenience, in the following we spell out some details of this structure. We do this mainly at the level of equivariant bundles.

The metric splitting of $\mathcal{G}$ is, of course, invariant and in this splitting the Riemannian mertic $g$ and the curvature $H$ are both invariant. Let $\varphi(e_a)=V_a+\xi_a$ in this splitting. Associated with the isotropic trivially extended $\mathfrak{g}$-action are two horizontal distributions on $M$ \cite{Ca}:
\begin{equation}\tau_\pm:=\{Y\in TM|g(Y, V_a)\pm \xi_a(Y)=0,\quad a=1,2,\cdots, \textup{dim} \mathfrak{g}\}.\label{tau}\end{equation}
They are just distributions derived by projecting $V_\pm^{red}$ to $TM$ and define two connections in the principal $G$-bundle $M\rightarrow M_{red}:=M/G$. They are basic for our later considerations. It is convenient to use $V_a^\pm:=V_a\pm g^{-1}\xi_a$. Denote the bundles generated by $\{V_a^\pm\}$ by $k_\pm\subset TM$ respectively. Then Eq.~(\ref{tau}) can be rephrased as the orthogonal decomposition
\begin{equation}TM=k_\pm\oplus \tau_\pm.\label{dec}\end{equation}

 Let $q: M\rightarrow M_{red}$ be the natural quotient map. Let us first interpret the short exact sequence properly:
\begin{equation}0\longrightarrow q^*(T^*M_{red})\stackrel{[\pi]^*}\longrightarrow K^\mathcal{G} \stackrel{[\pi]}\longrightarrow q^*(TM_{red}) \longrightarrow0,\label{exact}\end{equation}
where $q^*$ means the pull-back of vector bundles by the quotient map $q$, and $[\pi]$ denotes the composition $q_\ast\circ \pi$.
\begin{lemma}\[K^\mathcal{G}=\{Y+\eta\in \mathbb{T}M|g(Y, V_a)+g(\eta, \xi_a)=0, \xi_a(Y)+\eta(V_a)=0\},\]
and
\[\textup{ker}([\pi])=\{Y+\eta\in K^\mathcal{G}|Y\in \pi(K)\}.\]
\end{lemma}
\begin{proof}
This is obvious by definition of $K^\mathcal{G}$.
\end{proof}
For $E_{red}$ to be exact, $q^*(T^*M_{red})$ should be identified with $\textup{ker}([\pi])$ via the inner product on $K^\mathcal{G}$. This is realized as follows:
\begin{lemma}\label{ta}Let $T_{ab}:=g(V_a^+, V_b^+)=g(V_a^-, V_b^-)$ and denote its inverse by $T^{ab}$. Then for $\xi\in q^*(T^*M_{red})$,
\[[\pi]^*(\xi)=\xi-T^{ab}g(\xi, \xi_a)(V_b+\xi_b),\]
and the image of $[\pi]^*$ is precisely $\textup{ker}([\pi])$.
\end{lemma}
\begin{proof}
First note that
\[g(V_a^+, V_b^+)=g(V_a^-, V_b^-)=g(V_a, V_b)+g(\xi_a,\xi_b)\]
due to the fact $\iota_a\xi_b+\iota_b\xi_a=0$. This implies that the restriction of $g$ on either $\tau_+$ or $\tau_-$ gives rise to the same Riemannian metric on $TM_{red}$, just as required.

For $\xi\in q^*(T^*M_{red})$, $[\pi]^*(\xi)\in K^\mathcal{G}$ is characterized by \[\langle[\pi]^*(\xi), Y+\eta\rangle=\xi([\pi](Y+\eta))=\xi(q_*(Y)), \quad \forall Y+\eta\in K^\mathcal{G}.\] We can assume
$[\pi]^*(\xi)=\xi+f^b(V_b+\xi_b)$ for some constants $f^b$ to be determined. We thus have
\[g(\xi+f^b\xi_b, \xi_a)+g(f^bV_b, V_a)=0,\quad a,b=1,2,\cdots, \textup{dim}\mathfrak{g},\]
i.e.
\[[g(V_a,V_b)+g(\xi_a,\xi_b)]f^b=T_{ab}f^b=-g(\xi, \xi_a),\quad a,b=1,2,\cdots, \textup{dim}\mathfrak{g}.\]
This leads to our expression of $[\pi]^*(\xi)$.

It's obvious that $\textup{Ran}([\pi]^*)\subset \textup{ker}([\pi])$. Thus $\textup{Ran}([\pi]^*)=\textup{ker}([\pi])$ by dimensional reason.
\end{proof}
We have already used the metric splitting of $\mathcal{G}$ to identify $E$ with $\mathbb{T}M$. $\mathcal{G}$, when restricted on $K^\mathcal{G}$, also gives rise to an isotropic splitting $\mathcal{G}(\textup{ker}[\pi])$ of the sequence (\ref{exact}). Then $\mathcal{G}(\textup{ker}[\pi])/G$ is a splitting of $K^\mathcal{G}/G$. Let $Q_{ab}=g(V_a, V_b)$ and $Q^{ab}$ be its inverse.
\begin{proposition}If $E_{red}$ is identified with $K^\mathcal{G}/G$ and $T^*M_{red}$ with $\textup{ker}[\pi]/G$, then $\mathcal{G}(\textup{ker}[\pi])/G$ is the metric splitting of $E_{red}$.
\end{proposition}
\begin{proof}The following proof can only be viewed as a detailed analysis of the obvious conclusion. We only need to prove that in this splitting, $V_+^{red}/G$ is the graph of the reduced metric $\tilde{g}$ on $TM_{red}$, since other splittings will involve extra $B$-transforms.

Note that \[\mathcal{G}(\textup{ker}[\pi])=\{Y+\eta\in K^\mathcal{G}|g^{-1}\eta\in \pi(K)\}.\]
The projection of $\mathcal{G}(\textup{ker}[\pi])$ to $TM$ is a third horizontal distribution $\tau$ on $M$ as a principal $G$-bundle. For $Y+\eta\in \mathcal{G}(\textup{ker}[\pi])$, $\eta$ is uniquely determined by $Y$. We thus write $\eta_Y$ instead of $\eta$.

As observed in \cite{Ca}, a typical element in $V_+\cap K^\mathcal{G}$ is of the form
\[A=Y+g^{-1}\eta_Y+g(Y)+\eta_Y,\quad Y\in \tau.\]
Note that $Y+\eta_Y\in \mathcal{G}(\textup{ker}[\pi])$ and $g^{-1}\eta_Y+g(Y)\in \textup{ker}[\pi]$. This means, in the splitting determined by $\mathcal{G}(\textup{ker}[\pi])/G$, $A$ descends to
\[[Y]+g(Y)-Q^{ab}\eta_Y(V_b)\xi_a\in TM_{red}\oplus T^*M_{red},\]
since $g^{-1}\eta_Y=Q^{ab}\eta_Y(V_b)V_a$ and
\[g^{-1}\eta_Y=-Q^{ab}\eta_Y(V_b)\xi_a+Q^{ab}\eta_Y(V_b)(V_a+\xi_a).\]
Note that
\begin{eqnarray*}g(Y,V_c)-Q^{ab}\eta_Y(V_b)\xi_a(V_c)&=&g(Y,V_c)+Q^{ab}\eta_Y(V_b)\xi_c(V_a)
\\&=&g(Y,V_c)+\xi_c(g^{-1}\eta_Y)
\\&=& g(Y,V_c)+g(\eta_Y,\xi_c)\\
&=&0.\end{eqnarray*}
So $g(Y)-Q^{ab}\eta_Y(V_b)\xi_a$ does live in $T^*M_{red}$. Now we only need to check that on $\tau$ we have
\[q^*\tilde{g}([Y])=g(Y)-Q^{ab}\eta_Y(V_b)\xi_a,\quad \forall Y\in \tau.\]
Let $Z\in \tau$. Then on one side by definition of $\tilde{g}$ we have
\begin{eqnarray*}
q^*\tilde{g}([Y])(Z)&=&\tilde{g}([Y],[Z])=g(Y+g^{-1}\eta_Y,Z+g^{-1}\eta_Z))\\&=&g(Y,Z)+\eta_Y(Z)+\eta_Z(Y)+g(\eta_Y,\eta_Z)\\
&=&g(Y,Z)+g(\eta_Y,\eta_Z).
\end{eqnarray*}
On the other side,
\begin{eqnarray*}g(Y,Z)-Q^{ab}\eta_Y(V_b)\xi_a(Z)&=&g(Y,Z)+Q^{ab}\eta_Y(V_b)\eta_Z(V_a)\\
&=&g(Y,Z)+\eta_Z(g^{-1}\eta_Y)\\
&=&g(Y,Z)+g(\eta_Y,\eta_Z).\end{eqnarray*}
Hence the claim follows.
\end{proof}
\noindent\emph{Remark}. $\tau$ is the average of $\tau_\pm$ in the following sense: If $Y$ is the lift of $[Y]$ in $\tau$, then $Y\pm g^{-1}\eta_Y$ are the lifts of $[Y]$ in $\tau_\pm$ respectively.

 Let us compute the curvature $\tilde{H}$ of the metric splitting $\mathcal{G}(\textup{ker}[\pi])/G$. Note that $\eta$ is a map from $\tau$ to $T^*M$. Since $\mathcal{G}(\textup{ker}[\pi])$ is isotropic, we have $\eta_X(Y)+\eta_Y(X)=0$. Additionally, it can be easily obtained that $\eta_X(V_a)=-\xi_a(X)$. We can look for a 2-form $\gamma:TM\rightarrow T^*M$ such that its restriction on $\tau$ is precisely $\eta$. Let $\theta^a$ be the connection form determined by $\tau$. A choice of $\gamma$ is then
\[\gamma=-\frac{1}{2}Q^{ab}[\xi_b-\xi_b(V_c)\theta^c]\wedge [g(V_a)-Q_{ad}\theta^d]-\xi_a\wedge \theta^a+\frac{1}{2}\xi_b(V_a)\theta^a\theta^b.\]
\begin{proposition}If $TM_{red}$ is modeled on $\tau$, then the curvature of the metric splitting of $E_{red}$ is $\tilde{H}=(H+d\gamma)|_\tau$.
\end{proposition}
\begin{proof}Let $[X]$ here denote a vector field on $M_{red}$ represented by the invariant lift $X\in\Gamma(\tau)$. Then
\begin{eqnarray*}\tilde{H}([X],[Y],[Z])&=&\langle[X+\eta_X, Y+\eta_Y]_H, Z+\eta_Z\rangle\\
&=&\langle[X, Y]+L_X\eta_Y-\iota_Yd\eta_X+\iota_Y\iota_XH, Z+\eta_Z\rangle\\
&=&\eta_Z([X, Y])+\eta_Y([Z,X])+\eta_X([Y,Z])+X\eta_Y(Z)\\&+&Y\eta_Z(X)+Z\eta_X(Y)+H(X,Y,Z)\\
&=&\gamma(Z,[X, Y])+\gamma(Y,[Z,X])+\gamma(X,[Y,Z])\\
&+&X\gamma(Y,Z)+Y\gamma(Z,X)+Z\gamma(X,Y)+H(X,Y,Z)\\
&=&(H+d\gamma)(X,Y,Z),\end{eqnarray*}
as required.
\end{proof}
\noindent\emph{Remark}. The appearance of $\tilde{H}$ depends on which connection among $\tau$, $\tau_\pm$ is used to model $TM_{red}$. In \cite{Ca}, $\tau_+$ is used to do this. In the next section, we will carry out the same computation in a way different from that of \cite{Ca} .

To conclude this section, we clarify some notation for our later use. If $\mathcal{M}$ is a bigger manifold carrying an isotropic trivially extended $G$-action and $M$ is an invariant submanifold of $\mathcal{M}$, the Courant algebroid $\mathcal{E}$ on $\mathcal{M}$ can be directly pulled back to $M$ and the isotropic trivially extended $G$-action also descends. In fact, if $\varphi(e_a)=V_a+\xi_a$ on $\mathcal{M}$ in some splitting, then a natural splitting arises in the pull-back of $\mathcal{E}$ and $\varphi(e_a)=V_a|_M+i^*(\xi_a)$, where $i$ is the inclusion map. By abuse of notation, we will only write $\varphi(e_a)=V_a+\xi_a$ either on $\mathcal{M}$ or on $M$. The ambiguity will be clarified by the context.
\section{Bismut connections in metric reduction}\label{RBis}
The basic context of this section is the same as that of the former one, and we continue to use the notation there. We try to figure out how the Bismut connections $\tilde{\nabla}^\pm$ in $M_{red}$ are reduced from those in $M$. The curvature of $\tilde{\nabla}^-$ is a basic ingredient in \cite{Wang} to interpret metric reduction in the formalism of balanced topological field theories.

Our starting point is Eq.~(\ref{Hif}) where Bismut connections are expressed using Courant bracket. Since by the reduction procedure established in \cite{BCG1}, the Courant algebroid $E_{red}$ on $M_{red}$ can naturally be described in terms of the Courant algebroid $E$ on $M$, one can expect that the Bismut connections on $M_{red}$ could be described in terms of the Courant bracket on $M$. The two connections $\tau_\pm$ play a fundamental role in this investigation.

Note that $\tilde{g}$ is in fact defined by restricting $g$ on $\tau_+$ (or $\tau_-$). This is different from the ordinary case. Let $\tilde{\nabla}^-$ be the $-$-Bismut connection on $M_{red}$ and let $[X]$ denote a vector field on $M_{red}$ represented by an invariant lift $X$ on $M$. $X^\pm$ are used to denote the unique lifts of $[X]$ in $\tau_\pm$ respectively. Let $\varrho_-([X], [Y])$ denote the unique lift of $\tilde{\nabla}_{[X]}^-[Y]$ in $\tau_-$.
\begin{theorem}\label{Bism}$\varrho_-([X], [Y])$ is the projection of $\nabla_{X^+}^-Y^-$ to $\tau_-$ along $k_-$, namely
\begin{equation}\varrho_-([X], [Y])=\nabla_{X^+}^-Y^-+T^{ab}g(Y^-, \nabla_{X^+}^-V^-_b)V^-_a,\label{Bis}\end{equation}
where $T^{ab}$ is the inverse of $T_{ab}=g(V^-_a, V^-_b)$.
\end{theorem}
\begin{proof}According to Eq.~(\ref{Hif}), in the metric splitting of $E_{red}$,
\[\tilde{\nabla}_{[X]}^-[Y]-\tilde{g}(\tilde{\nabla}_{[X]}^-[Y])=[[X]+\tilde{g}([X]), [Y]-\tilde{g}([Y])]_{\tilde{H}}^-.\]
Due to the discussion in \S~\ref{sect3} the R.H.S. of the above equation can be computed using the corresponding invariant sections of $K^\mathcal{G}$, i.e.
\[[X^++g(X^+), Y^--g(Y^-)]_H^-.\]
It should be noted that $\Gamma^G(K^\mathcal{G})$ is not involutive under the Courant bracket. Involutivity can only hold up to addition of invariant section of $K$. Therefore,
\[[X^++g(X^+), Y^--g(Y^-)]_H=A_++A_-+N,\]
where $A_\pm\in \Gamma(V_\pm^{red})$ and $N=2c^a(V_a+\xi_a)$ for some functions $c^a$ to be determined. Of course we want to separate $A_-$ from the above expression because $\varrho_-([X], [Y])=\pi_-(A_-)$ where $\pi_-$ is the projection from $V_-$ to $TM$.

We already have
\[[X^++g(X^+), Y^--g(Y^-)]_H^-=A_-+N_-,\]
where $N_-=c^a(V_a-g(V_a)+\xi_a-g^{-1}\xi_a)$.
Hence,
\[A_-+N_-=\nabla_{X^+}^-Y^--g(\nabla_{X^+}^-Y^-).\]
Therefore,
\[\varrho_-([X], [Y])+c^a(V_a-g^{-1}\xi_a)=\varrho_-([X], [Y])+c^aV^-_a=\nabla_{X^+}^-Y^-.\]
Due to the orthogonal decomposition $TM=\tau_-\oplus k_-$, the above equation means that $\varrho_-([X], [Y])$ is actually the $\tau_-$-part of $\nabla_{X^+}^-Y^-$ w.r.t. this decomposition. We then find
\[c^a=T^{ab}g(\nabla_{X^+}^-Y^-, V_b^-)=-T^{ab}g(Y^-, \nabla_{X^+}^-V^-_b).\]
We finally obtain the formula as required.
\end{proof}
\noindent\emph{Remark}. The result is very similar to the ordinary case except that a different orthogonal decomposition is used. In particular, if $[Z]$ is another vector field on $M_{red}$, then $\tilde{g}(\tilde{\nabla}_{[X]}^-[Y],[Z])=g(\nabla_{X^+}^-Y^-, Z^-)$.

Now we can turn to the problem of expressing the curvature of $\tilde{\nabla}^-$ in terms of that of $\nabla^-$. Let $\theta_\pm^a$ be the connection forms associated to $\tau_\pm$ respectively and let $\Omega_\pm^a$ be the associated curvatures. For later use, we want to express $\Omega_\pm^a$ in terms of $V_a$, $\xi_a$. Let $K_{ab}=g_{ab}-\xi_a(V_b)$ and $K^{ba}$ its inverse, i.e. $K^{bc}K_{ab}=\delta_a^c$.

\begin{lemma}\label{cur}Let $\Omega_\pm^a$ be the curvatures of $\tau_\pm$. Then
\[\Omega_+^a|_{\tau_+}=K^{ba}d\xi^+_b|_{\tau_+},\quad \Omega_-^a|_{\tau_-}=K^{ab}d\xi^-_b|_{\tau_-},\]
where $\xi^\pm_b=g(V_b^\pm)$.
\end{lemma}
\begin{proof}We only compute $\Omega_+^a$. The computation for $\Omega_-^a$ is similar.
Note that
\[\theta_+=\theta_+^ae_a=t^{ba}g(V_b^+)e_a,\]
where $t^{ba}$ is to be determined. We have
\[t^{ba}g(V_b^+,V_c)=t^{ba}K_{cb}=\delta^a_c.\]
 Then $t^{ba}$ is precisely $K^{ba}$ and $\theta_+^a=K^{ba}g(V_b^+)$. Then
\begin{eqnarray*}
\Omega_+^a(X^+, Y^+)&=&d\theta_+^a(X^+, Y^+)=X^+\theta_+^a(Y^+)-Y^+\theta_+^a(X^+)-\theta_+^a([X^+, Y^+])\\
&=&-\theta_+^a([X^+, Y^+])=-K^{ba}(V_b^+,[X^+, Y^+])\\
&=&K^{ba}(d\xi^+_b)(X^+, Y^+).
\end{eqnarray*}
\end{proof}
Let $R^-$ and $\tilde{R}^-$ be the curvatures of $\nabla^-$ and $\tilde{\nabla}^-$ respectively. We have
\begin{theorem}\label{curvature}The curvature $\tilde{R}^-$ of $\tilde{\nabla}^-$ is
\begin{eqnarray*}&\quad&\tilde{g}(\tilde{R}^-([X], [Y])[Z], [W])=g(R^-(X^+, Y^+)Z^-, W^-)\\
&-&\frac{1}{2}K^{ab}(d\xi^+_a)(X^+, Y^+)(d\xi^-_b)(Z^-, W^-)\\&+&T^{ab}[g(Z^-, \nabla_{Y^+}^-V^-_a)g(W^-, \nabla_{X^+}^-V^-_b)-(X\leftrightarrow Y)],
\end{eqnarray*}
where $(X\leftrightarrow Y)$ denotes a term similar to the term in front of it, only with $X$ and $Y$ exchanged.
\end{theorem}
\begin{proof} Since $\tilde{\nabla}^-$ and $\nabla^-$ are metric connections, we have
\begin{eqnarray*}
&\quad&\tilde{g}(\tilde{\nabla}_{[X]}^-\tilde{\nabla}_{[Y]}^-[Z],[W])=[X]\tilde{g}(\tilde{\nabla}_{[Y]}^- [Z], [W])-\tilde{g}(\tilde{\nabla}_{[Y]}^-[Z], \tilde{\nabla}_{[X]}^-[W])\\&=&X^+g(\nabla^-_{Y^+}Z^-, W^-)\\
&-&g(\nabla^-_{Y^+}Z^-+T^{ab}g(Z^-,\nabla_{Y^+}^-V^-_b)V^-_a,\nabla^-_{X^+}W^-+T^{cd}g(W^-,\nabla_{X^+}^-V^-_d)V^-_c)\\
&=& g(\nabla_{X^+}^-\nabla_{Y^+}^-Z^-, W^-)-T^{ab}g(Z^-, \nabla_{Y^+}^-V^-_a)g(W^-, \nabla_{X^+}^-V^-_b)\\
&-&2T^{ab}g(Z^-,\nabla_{Y^+}^-V^-_b)g(V^-_a, \nabla^-_{X^+}W^-)\\
&=&g(\nabla_{X^+}^-\nabla_{Y^+}^-Z^-, W^-)+T^{ab}g(Z^-, \nabla_{Y^+}^-V^-_a)g(W^-, \nabla_{X^+}^-V^-_b),
\end{eqnarray*}
where Eq.~(\ref{Bis}) is used.

Additionally,
\begin{eqnarray*}\tilde{g}(\tilde{\nabla}_{[[X],[Y]]}^-[Z], [W])&=&g(\nabla^-_{[X^+, Y^+]+\Omega_+(X^+, Y^+)}Z^-, W^-)\\
&=&g(\nabla^-_{[X^+, Y^+]}Z^-, W^-)+g(\nabla^-_{\Omega_+(X^+, Y^+)}Z^-, W^-),\end{eqnarray*}
where we have used the identity\footnote{$\widetilde{[X]}$ denotes the horizontal lift of $[X]$ in $\tau_+$.}
\[[X^+, Y^+]-\widetilde{[[X], [Y]]}=-\Omega_+(X^+, Y^+)=-\Omega^a_+(X^+, Y^+)V_a,\]
and $\Omega^a_+$ is the curvature of $\tau_+$.

By Lemma.~\ref{cur}, we have
\[g(\nabla_{\Omega_+(X^+, Y^+)}^-Z^-, W^-)=K^{ba}(d\xi^+_b)(X^+, Y^+)g(\nabla_{V_a}^-Z^-, W^-).\]
Note that \begin{eqnarray*}g(\nabla_{V_a}^-Z^-, W^-)&=&g(\nabla_{V_a}Z^-,W^-)-\frac{1}{2}H(V_a, Z^-, W^-)\\
&=&g(\nabla_{V_a}Z^-,W^-)-\frac{1}{2}(d\xi_a)(Z^-, W^-),\end{eqnarray*}
\begin{eqnarray*}g(\nabla_{V_a}Z^-,W^-)&=&g(\nabla_{Z^-}V_a, W^-)=Z^-(g(V_a)(W^-))-g(V_a, \nabla_{Z^-}W^-)\\
&=&Z^-(g(V_a)(W^-))-g(V_a, \nabla_{W^-}Z^-)-g(V_a)([Z^-, W^-])\\
&=&Z^-(g(V_a)(W^-))-W^-(g(V_a)(Z^-))+g(\nabla_{W^-}V_a, Z^-)\\&-&g(V_a)([Z^-, W^-])\\
&=&dg(V_a)(Z^-, W^-)+g(\nabla_{W^-}V_a, Z^-),
\end{eqnarray*}
and
\[g(\nabla_{V_a}Z^-,W^-)+g(\nabla_{W^-}V_a, Z^-)=0.\]
Then we have
\begin{equation}g(\nabla_{V_a}^-Z^-, W^-)=\frac{1}{2}d\xi^-_a(Z^-, W^-).\label{3form}\end{equation}
Combining all the above ingredients together, we come to the conclusion.
\end{proof}

Similar formulae for $\tilde{\nabla}^+$ hold, but we just write down the counterpart of Eq.~(\ref{Bis}) for $\tilde{\nabla}^+$ for later use. The detail is left to the interested reader. Let $\varrho_+([X],[Y])$ be the $\tau_+$-lift of $\tilde{\nabla}^+_{[X]}[Y]$. Then
\begin{equation}\varrho_+([X],[Y])=\nabla^+_{X^-}Y^++T^{ab}g(Y^+, \nabla^+_{X^-}V_b^+)V_a^+.\label{bis}\end{equation}

As an application of our formula (\ref{bis}), we use it to compute the curvature $\tilde{H}$ of the reduced metric splitting again. 
\begin{proposition}\label{plus}The curvature $\tilde{H}$ of the reduced metric splitting is $(H+\Omega_+^a\wedge \xi_a)|_{\tau_+}$.
\end{proposition}
\begin{proof}
Since $\tilde{H}$ is the torsion of $\tilde{\nabla}^+$, we have
\begin{eqnarray*}\tilde{H}([X],[Y],[Z])&=&\tilde{g}(\tilde{\nabla}^+_{[X]}[Y], [Z])-\tilde{g}(\tilde{\nabla}^+_{[Y]}[X], [Z])-\tilde{g}([[X],[Y]],[Z])\\
&=&g(\nabla^+_{X^-}Y^+, Z^+)-g(\nabla^+_{Y^-}X^+, Z^+)\\&-&g([X^+, Y^+]+\Omega_+^a(X^+, Y^+)V_a, Z^+).\end{eqnarray*}
Since $X^-$ is uniquely determined by $X^+$, define $\varsigma(X^+)=X^+-X^-$. By definition, $g(X^+-\varsigma(X^+), V_a^-)=0$,
i.e.
\[g(V_b, V_a^-)\varsigma^b(X^+)=g(X^+, V_a^-)=-2\xi_a(X^+).\]
Since $g(V_b, V_a^-)=K_{ab}$, we find $\varsigma(X^+)=-2K^{ab}\xi_b(X^+)V_a$,
and
\begin{eqnarray*}g(\nabla^+_{\varsigma(X^+)}Y^+, Z^+)&=&-2K^{ab}\xi_b(X^+)g(\nabla^+_{V_a}Y^+, Z^+)\\
&=&-K^{ab}\xi_b(X^+)d\xi_a^+(Y^+,Z^+)\\
&=&-\Omega_+^b(Y^+,Z^+)\xi_b(X^+),
\end{eqnarray*}
where we have used a counterpart of Eq.~(\ref{3form}) for $\nabla^+$ and Lemma.~\ref{cur}. Thus we obtain
\begin{eqnarray*}
\tilde{H}([X],[Y],[Z])&=&H(X^+,Y^+,Z^+)-g(\nabla^+_{\varsigma(X^+)}Y^+, Z^+)\\&+&g(\nabla^+_{\varsigma(Y^+)}X^+, Z^+)+\Omega_+^a(X^+, Y^+)\xi_a(Z^+)\\
&=&H(X^+,Y^+,Z^+)+\Omega_+^b(Y^+,Z^+)\xi_b(X^+)\\
&-&\Omega_+^b(X^+,Z^+)\xi_b(Y^+)+\Omega_+^b(X^+,Y^+)\xi_b(Z^+)\\
&=&(H+\Omega_+^b\wedge \xi_b)(X^+,Y^+,Z^+),
\end{eqnarray*}
which recovers the result in \cite{Ca}.
\end{proof}
As a conclusion, we briefly discuss the metric reduction from a bigger manifold $\mathcal{M}$ to a submanifold $M$. The Courant algebroid $\mathcal{E}$ over $\mathcal{M}$ and the generalized metric $\mathcal{G}$ can be directly pulled back to $M$. This situation can be treated in the same spirit as before but is much simplified. If $M$ is locally defined by $\sigma^\alpha=0$, $\alpha=1,2,\cdots, \textup{dim} \mathcal{M}-\textup{dim} M$, one only needs to use $\{d\sigma^\alpha\}$ to generate the bundle $K$ on $M$. Then $K^\bot$ and $K^\mathcal{G}$ can be similarly defined. The metric splitting of $\mathcal{G}$ directly gives rise to the metric splitting on the reduced Courant algebroid. The reduced metric $\tilde{g}$ is just the restriction of $g$ on $TM$ and the curvature $\tilde{H}$ on $M$ is just the pull-back of $H$ on $\mathcal{M}$ by the inclusion map.

We still use $\tilde{\nabla}^-$ to denote the reduced $-$-Bismut connection. Let $G^{\alpha\beta}=g(d\sigma^\alpha, d\sigma^\beta)|_M$ and $G_{\alpha\beta}$ be its inverse.
\begin{proposition}Let $\bar{X}, \bar{Y}$ be vector fields on $M$, and $X$, $Y$ be their arbitrary extensions to $\mathcal{M}$. Then
\begin{equation}\tilde{\nabla}^-_{\bar{X}}\bar{Y}=\nabla_X^-Y|_M+G_{\alpha\beta}(Y, \nabla_X^- d\sigma^\beta)|_M(g^{-1}d\sigma^\alpha)|_M.\label{sbis2}\end{equation}
\end{proposition}
\begin{proof}We have an orthogonal decomposition $T\mathcal{M}|_M=TM\oplus Q$, where $Q$ is the normal bundle of $M$ in $\mathcal{M}$ and is locally generated by $\{g^{-1}d\sigma^\alpha|_M\}$. It is easy to check that $\tilde{\nabla}^-_{\bar{X}}\bar{Y}$ is just the projection of $\nabla_X^-Y|_M$ to $TM$ along $Q$. This is enough to lead to the conclusion.
\end{proof}

As for the curvature of $\tilde{\nabla}^-$, we have
\begin{proposition}The curvature of $\tilde{\nabla}^-$ is
\begin{eqnarray*}\bar{g}(\tilde{R}^-(\bar{X}, \bar{Y})\bar{Z}, \bar{W})&=&g(R^-(X, Y)Z, W)|_M\\&+&G_{\alpha\beta}[(Z, \nabla_Y^- d\sigma^\beta)(W, \nabla_X^- d\sigma^\alpha)-(X\leftrightarrow Y)]|_M,\end{eqnarray*}
where $(X\leftrightarrow Y)$ still denotes a term similar to the term in front of it, only with $X$ and $Y$ exchanged.
\end{proposition}
\begin{proof}
We leave the proof to the interested reader. A detailed argument can also be found in \cite{Wang}.
\end{proof}
\section{Metric generalized principal $G$-bundles and relative curvatures}\label{mgp}
In this section, motivated by former observations and also for later use, we investigate generalized principal $G$-bundles in the presence of an invariant generalized metric.

The notion of generalized principal bundles was introduced in \cite{Wang1} to define generalized holomorphic structures in the setting of principal bundles.

\begin{definition} A generalized principal $G$-bundle over $M$ is a triple $(P, E, \varphi)$ such that

\indent (i) $p:P\rightarrow M$ is an ordinary principal $G$-bundle,

\indent (ii) $E$ is a Courant algebroid over $P$ and $\varphi$ is an isotropic trivially extended $G$-action on $E$.
\end{definition}
In contrast with Def.~\ref{ext}, the notion of generalized principal bundles hardly contains any essentially new points, but emphasizes a different aspect of the same object. So $E$ descends to $M$ in the same way as before. In the following, we additionally assume that there is a $G$-invariant generalized metric $\mathcal{G}$ on $E$ and call $P$ a \emph{metric generalized principal bundle}. Then the two connections $\tau_\pm$ again arise. Let $\tilde{\nabla}^\pm$ denote the Bismut connections in the base manifold $M$.

\begin{definition}Let $X, Y$ be vector fields on $M$, and $X^+$, $Y^-$ be their lifts in $\tau_\pm$ respectively. The relative curvature of the pair $(\tau_+, \tau_-)$ is
\[R(X,Y)=(\tilde{\nabla}^-_XY)^--(\tilde{\nabla}^+_{Y}X)^+-[X^+, Y^-],\]
where $(\tilde{\nabla}^-_XY)^-$ is the $\tau_--$lift of $\tilde{\nabla}^-_XY$ and $(\tilde{\nabla}^+_{Y}X)^+$ is the $\tau_+$-lift of $\tilde{\nabla}^+_{Y}X$.
\end{definition}
It is not hard to check that $R$ is tensorial and takes values in the vertical distribution. There is a vector-bundle version of the notion of relative curvature.
\begin{definition}On a generalized Riemannian manifold $(M, g, H)$, if a vector bundle $W$ is equipped with two connections $\nabla^\pm$, then the \emph{relative curvature} of the pair $(\nabla^+, \nabla^-)$ is defined as
\[R(X,Y)s=\nabla^+_X\nabla^-_Ys-\nabla^-_Y\nabla^+_Xs-\nabla^-_{\tilde{\nabla}^-_XY}s+\nabla^+_{\tilde{\nabla}^+_YX}s, \quad \forall s\in \Gamma(W),\]
where $\tilde{\nabla}^\pm$ are the Bismut connections in the base manifold $M$.
\end{definition}
\noindent \emph{Remark}. It can be recognized that the relative curvature for a vector bundle is actually part of the curvature of a generalized connection defined in \cite{Gu1}: In the formula for the latter, simply by letting the two arguments take values in $V_+$ and $V_-$ respectively (recall that $V_\pm$ are the eigenbundles of the generalized metric $\mathcal{G}$), one recovers a relative curvature.

If $\rho:G\rightarrow \textup{End}(W_0)$ is a representation of $G$ in a vector space $W_0$, then $\tau_\pm$ in the metric generalized principal $G$-bundle $P$ give rise to two connections $\nabla^\pm$ in the associated vector bundle $W_0\times_\rho P$. It should be pointed out that since by our convention $G$ acts on $P$ from the left, $G$ should act on $W_0$ from the right; in particular, $\rho_*([v, w])=-[\rho_*(v), \rho_*(w)]$ for $v, w\in \mathfrak{g}$.
\begin{proposition}
If $R^aV_a$ is the relative curvature of the pair $(\tau_+, \tau_-)$ in the metric generalized principal $G$-bundle $P$, then $R^a\rho_\ast(e_a)$ is the relative curvature of the pair $(\nabla^+, \nabla^-)$ in the associated vector bundle $W_0\times_\rho P$.
\end{proposition}
\begin{proof}
Since the computation is essentially local, we can safely assume $P$ is of the form $G\times M$. Let $\theta_\pm$ be the connection form of $\tau_\pm$ respectively. Then
\[X^+=X-\theta^a_+(X)V_a,\quad Y^-=Y-\theta_-^a(Y)V_a.\]
Note that
\begin{eqnarray*}
\tilde{\nabla}^-_XY-\tilde{\nabla}^+_YX&=&\tilde{\nabla}_XY-\tilde{\nabla}_YX
-\frac{1}{2}g^{-1}\tilde{H}(X,Y)-\frac{1}{2}g^{-1}\tilde{H}(Y,X)\\
&=&\tilde{\nabla}_XY-\tilde{\nabla}_YX\\
&=&[X,Y],
\end{eqnarray*}
where $\tilde{\nabla}$ is the Levi-Civita connection on $M$ and $\tilde{H}$ the curvature of the reduced metric splitting.
Thus the relative curvature of the pair $(\tau_+, \tau_-)$ is
\begin{eqnarray*}R(X,Y)&=&[-\theta_-^a(\tilde{\nabla}^-_XY)+\theta_+^a(\tilde{\nabla}^+_{Y}X)]V_a
+[X, \theta_-^b(Y)V_b]\\&+&[\theta_+^b(X)V_b, Y]-[\theta_+^a(X)V_a, \theta_-^b(Y)V_b]\\
&=&[X\theta_-^a(Y)-Y\theta_+^a(X)-\theta_-^a(\tilde{\nabla}^-_XY)+\theta_+^a(\tilde{\nabla}^+_{Y}X)]V_a\\
&-&\theta_+^c(X)\theta_-^b(Y)f_{cb}^aV_a.\end{eqnarray*}

Let $s$ be the frame of $W_0\times_\rho P$ corresponding to the trivialization of $P$. We have
\[\nabla^+_X\nabla^-_Ys=\nabla^+_X(s\rho_*(\theta_-(Y)))=s(\rho_*(X\theta_-(Y)))+s\rho_*(\theta_+(X))\rho_*(\theta_-(Y)),\]
and
\[\nabla^-_Y\nabla^+_Xs=\nabla^-_Y(s\rho_*(\theta_+(X)))=s(\rho_*(Y\theta_+(X)))+s\rho_*(\theta_-(Y))\rho_*(\theta_+(X)).\]
Therefore, the relative curvature of the pair $(\nabla^+, \nabla^-)$ is
\begin{eqnarray*}
R(X,Y)&=&\rho_*(X\theta_-(Y))-\rho_*(Y\theta_+(X))-\rho_*(\theta_-(\tilde{\nabla}^-_XY))+\rho_*(\theta_+(\tilde{\nabla}^+_YX))\\
&+&\rho_*(\theta_+(X))\rho_*(\theta_-(Y))-\rho_*(\theta_-(Y))\rho_*(\theta_+(X))\\
&=&\rho_*(X\theta_-(Y))-\rho_*(Y\theta_+(X))-\rho_*(\theta_-(\tilde{\nabla}^-_XY))+\rho_*(\theta_+(\tilde{\nabla}^+_YX))\\
&-&\rho_*([\theta_+(X),\theta_-(Y)]).
\end{eqnarray*}
The claim then follows.
\end{proof}
Now let us go back to the context of \S~\ref{RBis} and view $M$ as a metric generalized principal $G$-bundle over $M_{red}$. We want to derive a formula for the relative curvature of the pair $(\tau_+, \tau_-)$ in terms of the data of the isotropic trivially extended $G$-action.

By Eq.~(\ref{Bis}) and Eq.~(\ref{bis}),
\[\varrho_-([X],[Y])=\nabla^-_{X^+}Y^-+T^{ab}g(Y^-,\nabla^-_{X^+}V_b^-)V_a^-,\]
and
\[\varrho_+([Y], [X])=\nabla^+_{Y^-}X^++T^{ab}g(X^+,\nabla^+_{Y^-}V_b^+)V_a^+.\]
Therefore, the relative curvature is
\begin{eqnarray*}R([X],[Y])&=&T^{ab}[g(Y^-,\nabla^-_{X^+}V_b^-)V_a^--g(X^+,\nabla^+_{Y^-}V_b^+)V_a^+]\\
&=&T^{ab}[g(Y^-,\nabla^-_{X^+}V_b^-)-g(X^+,\nabla^+_{Y^-}V_b^+)]V_a\\
&-&T^{ab}[g(Y^-,\nabla^-_{X^+}V_b^-)+g(X^+,\nabla^+_{Y^-}V_b^+)]g^{-1}\xi_a.\end{eqnarray*}
A simple calculation shows $g(Y^-,\nabla^-_{X^+}V_b^-)+g(X^+,\nabla^+_{Y^-}V_b^+)=0$. We finally have
\begin{equation}R^a([X],[Y])=-2T^{ab}g(\nabla^-_{X^+}Y^-,V_b^-).\end{equation}

If additionally $M$ together with its structure of a metric generalized principal $G$-bundle comes as an invariant submanifold of a bigger manifold $\mathcal{M}$, which carries an isotropic trivially extended $G$-action and a $G$-invariant generalized metric, we can express the above result in terms of extensions $\breve{X}$, $\breve{Y}$ of $X^+$, $Y^-$ on $\mathcal{M}$. Let $\breve{g}$ be the metric on $\mathcal{M}$ and $\breve{\nabla}^\pm$ be Bismut connections on $\mathcal{M}$. If $M\subset \mathcal{M}$ is locally defined by $\mu^\alpha=0$ for $\alpha=1,2,\cdots, \textup{dim}\mathcal{M}-\textup{dim}M$, then by Eq.~(\ref{sbis2}), we have
\begin{equation}
R^a([X],[Y])=-2T^{ab}\breve{g}(\breve{\nabla}^-_{\breve{X}}\breve{Y}, V_b^-)|_M+2T^{ab}G_{\alpha\beta}d\mu^\alpha (\breve{\nabla}^-_{\breve{X}}\breve{Y})|_M d\mu^\beta(V_b^-)|_M,\label{rce}
\end{equation}
where $G_{\alpha\beta}$ is the inverse of $G^{\alpha\beta}=\breve{g}(d\mu^\alpha, d\mu^\beta)|_M$. This formula will be crucial in \S~\ref{ghs}.

\section{Generalized K$\ddot{a}$hler reduction}\label{GKR}
In the framework of \cite{BCG1} or \cite{Ca}, the reduction of a $G$-invariant generalized K$\ddot{a}$hler manifold $\mathcal{M}$ involves two stages: (i) a $G$-invariant submanifold $M\subset \mathcal{M}$ is singled out, possibly by the zero-level set of an equivariant map $\mu: \mathcal{M}\rightarrow \mathfrak{h}^*$, where $\mathfrak{h}^*$ is the dual of a $\mathfrak{g}$-module $\mathfrak{h}$. At the same time, the bundle $K$ over $M$, locally generated by $\{V_a+\xi_a\}$ and $\{d\mu^\alpha\}$, is constructed. Then $K^\bot$ is again defined as the orthogonal complement of $K$ in $\mathbb{T}\mathcal{M}|_M$ and one gets the important bundle $K^\mathcal{G}=K^\bot \cap \mathcal{G}(K^\bot)$ over $M$. (ii) If $\mathbb{J}_1$ preserves $K^\mathcal{G}$, i.e.,
\begin{equation}\mathbb{J}_1K^\mathcal{G}=K^\mathcal{G},\label{kcon}\end{equation}
then $K^\mathcal{G}/G$ naturally acquires two complex structures. Since $K^\mathcal{G}/G$ can be identified with $E_{red}$, these are actually almost generalized complex structures on $M_{red}:=M/G$. Integrability of these structures stems from the general reduction theory of Dirac structures.

We prefer to put things in another way: One can first realize the metric reduction from $\mathcal{M}$ to $M$. With this in place, we are in the situation of \S~\ref{sect3} and can then realize the metric reduction from $M$ to $M_{red}$. Now as before, there are two connections $\tau_\pm$ in $M$ as a metric generalized principal $G$-bundle. Then Eq.~(\ref{kcon}) simply means $J_\pm \tau_\pm=\tau_\pm$, i.e. $J_\pm$ preserve the two distributions on $M$ respectively, where $J_\pm$ are the underlying complex structures on $\mathcal{M}$.

Although integrability of the reduced generalized K$\ddot{a}$hler structure in $M_{red}$ is almost obvious from the more general viewpoint, we would like to provide another proof of this fact, which fits in well with our viewpoint towards metric reduction. This approach is a bit more complicated and indirect, but may shed some new light on generalized K$\ddot{a}$hler reduction. Note that in the following, a Courant algebroid $\mathcal{E}$ on $\mathcal{M}$ carrying an isotropic trivially extended $G$-action is understood as the basic background.
\begin{theorem}(\cite{BCG1} \cite{Ca}) Let $\mathcal{M}$ be a $G$-invariant generalized K$\ddot{a}$hler manifold and $M$ a $G$-invariant submanifold such that $J_\pm \tau_\pm=\tau_\pm$. Then the generalized K$\ddot{a}$hler structure descends to the reduced manifold $M_{red}$.
\end{theorem}
\begin{proof}As the reduced generalized K$\ddot{a}$hler structure is well understood in the literature, we only pay attention to the integrability condition.

Since $TM_{red}$ is modeled on both $\tau_+$ and $\tau_-$ on $M$, thus $J_\pm \tau_\pm=\tau_\pm$ implies that $M_{red}$ has two almost complex structures $\tilde{J}_\pm$. The compatibility of $\tilde{J}_\pm$ with the reduced metric $\tilde{g}$ is also obvious. In this situation, according to \cite{Gu00}, to obtain the conclusion, we need to prove (i) $\tilde{\nabla}^{\pm}\tilde{J}_\pm=0$ and (ii) $\tilde{H}$ is of type $(2,1)+(1,2)$ w.r.t. both $\tilde{J}_\pm$. Note that here $\tilde{\nabla}^{\pm}$ are the Bismut connections on $M_{red}$ and $\tilde{H}$ is the curvature of the reduced metric splitting.

W use $\breve{X}$ to denote an extension of $X^+\in \Gamma(\tau_+)$ or $X^-\in\Gamma(\tau_-)$ to $\mathcal{M}$ and the Bismut connections in $\mathcal{M}$ are denoted by $\breve{\nabla}^\pm$. By Eq.~(\ref{bis}) and an analogue of Eq.~(\ref{sbis2}) for $\breve{\nabla}^+$, we have
\begin{eqnarray*}(\tilde{\nabla}^+_{[X]}\tilde{J}_+[Y], [Z])&=&(\nabla^+_{X^-}J_+Y^+, Z^+)=(\breve{\nabla}^+_{\breve{X}}J_+\breve{Y}, \breve{Z})|_M\\
&=&(J_+\breve{\nabla}^+_{\breve{X}}\breve{Y}, \breve{Z})|_M=-(\breve{\nabla}^+_{\breve{X}}\breve{Y}, J_+\breve{Z})|_M\\
&=&-(\nabla^+_{X^-}Y^+, J_+Z^+)=-(\tilde{\nabla}^+_{[X]}[Y],\tilde{J}_+[Z])\\
&=& (\tilde{J}_+\tilde{\nabla}^+_{[X]}[Y],[Z]),\end{eqnarray*}
where the fact $\breve{\nabla}^+ J_+=0$ is used. We thus have proved that $\tilde{\nabla}^+ \tilde{J}_+=0$. $\tilde{\nabla}^- \tilde{J}_-=0$ can be proved similarly.

To see that $\tilde{H}$ is of type $(2,1)+(1,2)$ w.r.t. $\tilde{J}_+$, since by Prop.~\ref{plus}
\[\tilde{H}([X], [Y], [Z])=(\check{H}|_M+\Omega_+^a\wedge \xi_a)(X^+,Y^+,Z^+),\]
and the curvature $\check{H}$ of the metric splitting of $\mathcal{E}$ is of type $(2,1)+(1, 2)$ w.r.t. $J_+$, we only need to prove $\Omega_+^a$ is of type $(1,1)$, namely
\[\Omega_+^a(X^+, Y^+)=0\]
if $X^+, Y^+\in (\tau_+\otimes \mathbb{C})\cap T_{1,0}^+\mathcal{M}|_M$. Note that $\Omega_+^a=K^{ba}d\xi_b^+$, we should prove $d\xi_a^+(X^+, Y^+)=0$. In fact,
\begin{eqnarray*}
d\xi_a^+(X^+, Y^+)&=&[dg(V_a)+d\xi_a](X^+, Y^+)
=X^+g(V_a,Y^+)-Y^+g(V_a, X^+)\\&-&g(V_a, [X^+, Y^+])+\breve{H}|_M(V_a, X^+, Y^+)\\
&=&\sqrt{-1}X^+\omega_+|_M(V_a^A, Y^+)-\sqrt{-1}Y^+\omega_+|_M(V_a^A, X^+)\\
&-&\sqrt{-1}\omega_+|_M(V_a^A, [X^+, Y^+])+\breve{H}|_M(V_a^A, X^+, Y^+),
\end{eqnarray*}
where $V_a^A$ is the $T_{0,1}^+\mathcal{M}$-part of $V_a$ and we have used the following two facts: (i) $\breve{g}$ and $J_+$ are compatible and $\omega_+=\breve{g}J_+$; (ii) $\breve{H}$ is of type $(2,1)+(1,2)$ w.r.t. $J_+$. Consequently we have
\begin{eqnarray*}d\xi_a^+(X^+, Y^+)&=&-\sqrt{-1}d(\omega_+|_M)(X^+, Y^+, V_a^A)+\breve{H}|_M(V_a^A, X^+, Y^+)\\
&=&-\sqrt{-1}(d\omega_+)|_M(X^+, Y^+, V_a^A)+\breve{H}|_M(X^+, Y^+, V_a^A)\\
&=&-(d\omega_+)|_M(J_+X^+, J_+Y^+, J_+V_a^A)+\breve{H}|_M(X^+, Y^+, V_a^A)\\
&=&(d_+^c\omega_++\breve{H})|_M(X^+, Y^+, V_a^A)\\
&=&0,\end{eqnarray*}
where we have used the identity $d_+^c\omega_++\breve{H}=0$ on $\mathcal{M}$. Similarly, the curvature $\Omega_-^a$ of $\tau_-$ is of type $(1,1)$ w.r.t. $\tilde{J}_-$ and $\tilde{H}$ is of type $(2,1)+(1,2)$ w.r.t. $\tilde{J}_-$.
\end{proof}
 \noindent\emph{Remark}. The proof has some byproducts. It implies that $M$ as a principal $G$-bundle over generalized K$\ddot{a}$hler manifold $M_{red}$ carries two connections $\tau_\pm$ whose curvatures are of type $(1,1)$ w.r.t. $\tilde{J}_\pm$ respectively. Thus \emph{any associated complex vector bundle $W$ naturally has a biholomorphic structure}, i.e. $W$ is holomorphic simultaneously w.r.t. both of $\tilde{J}_\pm$. Such vector bundles play a basic role in the work \cite{Hu} to find an analogue of Hermite-Einstein equations in the context of biHermitian manifolds. Thus generalized K$\ddot{a}$hler reduction actually provides examples of biholomorphic structures.
\section{Generalized holomorphic structures from generalized K$\ddot{a}$hler reduction}\label{ghs}
Generalized holomorphic vector bundles are analogues of holomorphic vector bundles in complex geometry. Due to the observation in the end of \S~\ref{GKR}, it is natural to ask whether generalized holomorphic vector bundles could arise as byproducts of generalized K$\ddot{a}$hler reduction. The goal of this section is mainly to provide some examples to give an affirmative answer to this question. We will continue to use notation in \S~\ref{GKR}. Note that $M\subset\mathcal{M}$ is actually a metric generalized principal $G$-bundle, carrying the pair $(\tau_+, \tau_-)$ of connections. Let $\tau_\pm^{0,1}$ denote the $-\sqrt{-1}$-eigensubbundles of $\tau_\pm\otimes \mathbb{C}$ w.r.t. $J_\pm$ respectively.
\begin{lemma}Any associated complex vector bundle $W$ of $M$ as a principal $G$-bundle is naturally generalized holomorphic if the relative curvature $R$ of the pair $(\tau_+, \tau_-)$ satisfies
\begin{equation}R([X],[Y])=0, \quad \forall [X]\in T_{0,1}^+M_{red}, \quad [Y]\in T_{0,1}^-M_{red}.\label{rc}\end{equation}
\end{lemma}
\begin{proof}Let $\nabla^{W\pm}$ be the connections in $W$ determined by $\tau_\pm$ in $M$. They can be combined to give a generalized connection $\mathcal{D}$ in $W$ in the sense of \cite{Gu1}. In fact, Since $A\in\Gamma(E_{red})$ can be uniquely written as
\[A=[X]+\tilde{g}([X])+[Y]-\tilde{g}([Y])\]
for some $[X],[Y]\in \Gamma(TM_{red})$ due to the decomposition $E_{red}=V_+^{red}/G\oplus V_-^{red}/G$, we can define
\[\mathcal{D}_A s=\nabla^{W+}_{[X]}s+\nabla^{W-}_{[Y]}s, \quad s\in \Gamma(W).\]
$\mathcal{D}$ can be additionally decomposed according to the decomposition $E_{red}\otimes \mathbb{C}=L_1^{red}\oplus \bar{L}_1^{red}$, where $L_1^{red}$ is the $\sqrt{-1}$-eigenbundle of the reduced generalized complex structure $\mathbb{J}_1^{red}$ of $\mathbb{J}_1$. Let $\bar{\partial}$ be the $\bar{L}_1^{red}$-part of $\mathcal{D}$. $\bar{\partial}$ can be further decomposed as $\bar{\partial}=\bar{\delta}_++\bar{\delta}_-$ where actually $\bar{\delta}_\pm$ are just the natural $\tilde{J}_\pm$-holomorphic structures in $W$ induced from $\nabla^{W\pm}$. Thus $\bar{\partial}$ is a generalized holomorphic structure iff Eq.~(\ref{ghc}) is satisfied. It is not hard to find this is exactly Eq.~(\ref{rc}).
\end{proof}
To provide concrete examples, we specify to the case of Hamiltonian generalized K$\ddot{a}$hler manifolds introduced by Lin and Tolman in \cite{LT}. Let $\mathcal{M}$ be a $G$-invariant generalized K$\ddot{a}$hler manifold. The extended $G$-action is called Hamiltonian if there is an equivariant map $\mu: \mathcal{M}\rightarrow \mathfrak{g}^*$ ($\mathfrak{g}^*$ carries the coadjoint action) such that
\begin{equation}\mathbb{J}_2 (V_a+\xi_a)=d\mu_a,\quad a=1,2,\cdots, \textup{dim}{\mathfrak{g}}\label{ham}\end{equation}
 where $\mu_a=\mu(e_a)$. According to Lemma 4.2 in \cite{Wang2}, in terms of the biHermitian data, Eq.~(\ref{ham}) is equivalent to
 \begin{equation}
 J_+V_a^+=J_-V_a^-=-\breve{g}^{-1}d\mu_a.\label{Ham1}
 \end{equation}
 If $G$ acts freely on $M=\mu^{-1}(0)$, then Eq.~(\ref{kcon}) naturally follows and $M_{red}=\mu^{-1}(0)/G$ carries a reduced generalized K$\ddot{a}$hler structure \cite{LT} .

Recall that we use $\breve{X}\in \Gamma(T\mathcal{M}\otimes \mathbb{C})$ to denote an extension of vector field $X^+\in \Gamma(\tau_+\otimes \mathbb{C})$ or $X^-\in\Gamma(\tau_-\otimes \mathbb{C})$ on $M$.
\begin{theorem}\label{hrd}
Assume the extended $G$-action on generalized K$\ddot{a}$hler manifold $\mathcal{M}$ is Hamiltonian. Then Eq.~(\ref{rc}) is satisfied if
\begin{equation}
d\mu(\breve{\nabla}^-_{\breve{X}}\breve{Y})|_M=0, \label{hrc}
\end{equation}
for any $X^+\in \Gamma(\tau_+^{0,1})$ and $Y^-\in \Gamma(\tau_-^{0,1})$.
\end{theorem}
\noindent\emph{Remark}. Eq.~(\ref{hrc}) means $\breve{\nabla}^-_{\breve{X}}\breve{Y}$ should be tangent to $M$.
\begin{proof}
As the result is obviously independent of which extensions we choose, we can safely assume that $\breve{X}\in \Gamma(T_{0,1}^+\mathcal{M})$ and $\breve{Y}\in \Gamma(T_{0,1}^-\mathcal{M})$.

Due to Eq.~(\ref{Ham1}), for any $Z\in \Gamma(T_{0,1}^-\mathcal{M})$ we have \[\breve{g}(Z, V_a^-)=\sqrt{-1}d\mu_a(Z).\]
Since $\breve{\nabla}^-J_-=0$, we have $\breve{\nabla}^-_{\breve{X}}\breve{Y}\in \Gamma(T_{0,1}^-\mathcal{M})$ and thus
\[\breve{g}(\breve{\nabla}^-_{\breve{X}}\breve{Y},V_a^-)=\sqrt{-1}d\mu_a(\breve{\nabla}^-_{\breve{X}}\breve{Y}).\]
Substituting this result in Eq.~(\ref{rce}), we find $R^a([X],[Y])$ is of the form
\[R^a([X],[Y])=-2\sqrt{-1}T^{ab}(\delta_b^c+\sqrt{-1}S_b^c)d\mu_c(\breve{\nabla}^-_{\breve{X}}\breve{Y})|_M,\]
for some real-valued functions $S_b^c$ on $M$. The conclusion immediately follows.
\end{proof}

We are now in a position to construct some generalized holomorphic bundles by generalized K$\ddot{a}$hler reduction. Here we content ourselves with a family of generalized holomorphic line bundles on $\mathbb{C}P^2$ equipped with certain non-trivial generalized K$\ddot{a}$hler structures. We follow the ideas of \cite{Gu00} \cite{LT} to deform the standard K$\ddot{a}$hler structure on $\mathbb{C}^3$ (as a generalized K$\ddot{a}$hler manifold) to new $S^1$-invariant generalized K$\ddot{a}$hler structures while keeping the standard symplectic structure fixed. Applying the Marsden-Weinstein reduction to the symplectic structure then gives non-trivial generalized K$\ddot{a}$hler structures on $\mathbb{C}P^2$. By choosing the deformations \emph{properly}, we would obtain metric generalized principal $S^1$-bundles over $\mathbb{C}P^2$ whose associated line bundles are generalized holomorphic.

Let us recall the deformation theory of \cite{Gu00} in some detail. Given a generalized K$\ddot{a}$hler structure on $\mathcal{M}$ and $\epsilon\in \Gamma(\wedge^2 \bar{L}_1)$, define $L_1^\epsilon=\{X+\iota_X\epsilon|X\in L_1\}$. For $\epsilon$ small enough, $L_1^\epsilon$ is an almost generalized complex structure. The integrability condition of this deformation is the Maurer-Cartan equation
\begin{equation}d_{L_1}\epsilon+\frac{1}{2}[\epsilon, \epsilon]_S=0,\label{MC}\end{equation}
where $[\cdot, \cdot]_S$ is the Schouten bracket induced from the Lie algebroid $L_1$. Note that since $L_1=L_+\oplus L_-$ and $L_2=L_+\oplus \bar{L}_-$, to keep $L_2$ fixed we should take $\epsilon\in\Gamma(\bar{L}_+\otimes \bar{L}_-)$.

\textbf{The standard K$\ddot{a}$hler structure on $\mathbb{C}^3$.}
Let $\mathcal{M}=\mathbb{C}^3$ with its canonical complex structure $J$ and K$\ddot{a}$hler structure $\omega=\sqrt{-1}\sum_{i=0}^3 dz_i\wedge d\bar{z}_i$. Let $E_i=\partial_{z_i}+d\bar{z}_i$, $F_i=\partial_{z_i}-d\bar{z}_i$, $i=0, 1, 2$. We have two complex vector bundles $L_+=\textup{span}\{\bar{E}_i\}$ and $L_-=\textup{span}\{\bar{F}_i\}$ over $\mathcal{M}$. Viewed as a generalized K$\ddot{a}$hler manifold, $\mathcal{M}$ has $L_1=L_+\oplus L_-$ and $L_2=L_+\oplus \bar{L}_-$ as the associated two generalized complex structures. A pure spinor of $L_1$ is $\varphi_1=dz_0dz_1dz_2$ and a pure spinor of $L_2$ is $\varphi_2=e^{-\sqrt{-1}\omega}$. Note that in the present setting $d_{L_1}$ is just the classical Dolbeault operator associated to $J$ and from the biHermitian viewpoint, we have chosen $J_+=J_-=J$.

$S^1$ acts on $\mathcal{M}$ by scaling:
\[e^{i\theta}\cdot (z_0,z_1,z_2)=(e^{i\theta}z_0,e^{i\theta}z_1, e^{i\theta}z_2 ).\] The infinitesimal action of $S^1$ is generated by the vector field
\[\partial_\theta=\sqrt{-1}\sum_{i=0}^2 (z_i\partial_{z_i}-\bar{z}_i\partial_{\bar{z}_i}).\]
Then $\mu=\sum_{i=0}^2 |z_i|^2-1$ is a moment map. By K$\ddot{a}$hler reduction, $\mu^{-1}(0)/S^1$ is a K$\ddot{a}$hler manifold. This is the canonical K$\ddot{a}$hler structure on the projective plane $\mathbb{C}P^2$.

\textbf{Deformations of the K$\ddot{a}$hler structure.}
 We choose \[\epsilon=\frac{1}{2}(\sum_{i=0}^2f_iE_i)\wedge (\sum_{j=0}^2 g_jF_j),\] where $f_i, g_i$ are homogeneous polynomials of $z_0, z_1$ and $z_2$ to be determined. Since $f_i, g_i$ are holomorphic, the integrability condition (\ref{MC}) accounts to the following equations:
\begin{equation}\left \{
\begin{array}{ll}
\sum_{p=0}^2f_p(g_k\partial_{z_p}g_q-g_q\partial_{z_p}g_k)=0,\\
\sum_{p=0}^2g_p(f_k\partial_{z_p}f_q-f_q\partial_{z_p}f_k)=0.\end{array}
 \right.
 \label{MC1}\end{equation}
 There are many solutions to these equations. We list two as follows: (i) $g_0=g_1=f_1=f_2=0$, $g_2=1$ and $f_0=z_0^2$; (ii) $g_i=1$, $i=0,1, 2$ and \begin{equation}f_0=(z_1-z_0)(z_2-z_0),\quad f_1=(z_0-z_1)(z_2-z_1),\quad f_2=(z_0-z_2)(z_1-z_2).\label{sol}\end{equation}
 Note that deformations associated to the two solutions are both $S^1$-invariant. Since we are only concerned with the behavior of $\epsilon$ over $M=\mu^{-1}(0)=S^5$, we can simply multiply $\epsilon$ by a nonzero complex number $\lambda$ such that $|\lambda|$ is small enough and then $L_1^\epsilon$ and $L_2$ together define a generalized K$\ddot{a}$hler structure on a bounded neighbourhood of $M$ in $\mathbb{C}^3$.

 In the following, for simplicity, we will set $g_0=g_1=g_2=1$. In this case, the first equation of (\ref{MC1}) holds trivially and the second equation can be written in a more compact form:
 \[\sum_{p=0}^2\partial_{z_p}(\frac{f_q}{f_k})=0.\]
 Thus we can choose $f_i$ to be functions of $z_1-z_0$ and $z_2-z_0$, e.g. our solution (ii) is such a choice. To make $\epsilon$ be $S^1$-invariant, we additionally require $f_i$ to be of degree 2. Now we have
 \[L_+^\epsilon=\textup{span}\{\bar{E}_i+f_i\sum_{p=0}^2 F_p\},\quad L_-^\epsilon=\textup{span}\{\bar{F}_i+\sum_{p=0}^2f_pE_p\}, \]
 and $L_1^\epsilon=L_+^\epsilon\oplus L_-^\epsilon$, $L_2=L_+^\epsilon\oplus \bar{L}_-^\epsilon$. Let $J_\pm^\epsilon$ be the underlying complex structures. Accordingly,
 \[T_{0,1}^+\mathcal{M}=\textup{span}\{\partial_{\bar{z}_i}+f_i\sum_{p=0}^2\partial_{z_p}\},\quad T_{0,1}^-\mathcal{M}=\textup{span}\{\partial_{\bar{z}_i}+\sum_{p=0}^2f_p\partial_{z_p}\}.\]

 It should be pointed out that generally with the above form of $L_1^\epsilon$ and $L_2$, we are not in the metric splitting, but this won't bother us much. Though Eq.~(\ref{Hif}) is written in the metric splitting, other splittings are equally fine, because $B$-transforms won't affect the tangent part of the equation, which is essential to obtain Bismut connections.

\textbf{The reduced generalized K$\ddot{a}$hler structure.}
We content ourselves with a glance at the reduced generalized K$\ddot{a}$hler structure, for we are more interested in the generalized holomorphic line bundles produced by the reduction procedure. The reduced generalized K$\ddot{a}$hler structure can be described conveniently in terms of pure spinors. We refer the interested reader to \cite{Du} for a detailed account of pure spinors in the setting of generalized reduction.

A pure spinor of $L_1^\epsilon$ is $\varphi_1^\epsilon=e^{-\epsilon}\cdot \varphi_1$. To find a pure spinor for its reduction $\tilde{\mathbb{J}}_1^\epsilon$ on $\mathbb{C}P^2$, roughly speaking, one simply pulls back $\varphi_1^\epsilon$ to $M=S^5$ and then pushes it forward to the quotient $M/S^1$ (i.e. contraction with $\partial_\theta$ on $M$). The two stages can be exchanged. Thus one first contracts $\varphi_1^\epsilon$ with $\partial_\theta$ on $\mathcal{M}$ and then pulls back the result to $S^5$. By 'roughly', we mean actually before pulled back, $\iota_{\partial_\theta}\varphi_1^\epsilon$ should be scaled to be $S^1$-invariant. The scaling cannot be realized globally, and when it breaks down, type-jumping of $\tilde{\mathbb{J}}_1^\epsilon$ occurs. A generic point on $\mathbb{C}P^2$ is of type 0 for $\tilde{\mathbb{J}}_1^\epsilon$, and for a point on the type-jumping locus, the type jumps to 2. A similar and more detailed analysis of some reduced generalized K$\ddot{a}$hler structures on $\mathbb{C}P^2$ can be found in \cite{BCG1}. Our argument above is along the same line of \cite{BCG1}.

 To determine the type-jumping locus of $\tilde{\mathbb{J}}_1^\epsilon$, we take the 0-form component of $\iota_{\partial_\theta}\varphi_1^\epsilon$. This produces a section $\rho$ of the dual bundle of the pure spinor line bundle of $\tilde{\mathbb{J}}_1^\epsilon$. The zero-locus of $\rho$ is precisely the type-jumping locus, which is singled out by the following homogeneous equation of degree 3:
\[z_0(f_1-f_2)+z_1(f_2-f_0)+z_2(f_0-f_1)=0.\]
This can also be found using the type formula in \cite{LT}. For example, if we choose the solution (\ref{sol}) the type-jumping locus consists of three lines in $\mathbb{C}P^2$: $z_0=z_1$, $z_0=z_2$ and $z_1=z_2$ with a joint-point $[1:1:1]$. The type-jumping locus can be viewed as a degenerate elliptic curve.

 \textbf{Generalized holomorphic structures from reduction.} We now check that $M=S^5$ as a metric generalized principal $S^1$-bundle really gives its associated line bundles a generalized holomorphic structure, i.e., Eq.~(\ref{hrc}) does hold for the present setting.

 Let us find $\tau_\pm^{0,1}$ first. Due to the proof of Thm.~\ref{hrd}, we only need to find $Z$ in $T_{0,1}^\pm\mathcal{M}$ such that $d\mu(Z)=0$. Denote $z=z_0+z_1+z_2$ and $F=F_0+F_1+F_2$. The image of $g$ on $\tau_+^{0,1}$ is spanned by
 \[A_1:=-(z_1+f_1\bar{z})(\bar{E}_0+f_0F)+(z_0+f_0\bar{z})(\bar{E}_1+f_1F)\]
 and
 \[A_2:=-(z_2+f_2\bar{z})(\bar{E}_0+f_0F)+(z_0+f_0\bar{z})(\bar{E}_2+f_2F)\]
 on $S^5$. Denote $h=f_0\bar{z}_0+f_1\bar{z}_1+f_2\bar{z}_2$ and $C=f_0E_0+f_1E_1+f_2E_2$. The image of $-g$ on $\tau_-^{0,1}$ is spanned by
  \[B_1:=-(z_1+h)(\bar{F}_0+C)+(z_0+h)(\bar{F}_1+C)\]
  and
  \[B_2:=-(z_2+h)(\bar{F}_0+C)+(z_0+h)(\bar{F}_2+C)\]
  on $S^5$. By abuse of notation, we also use $A_i$, $B_i$ to denote their extensions on $\mathcal{M}$ with the same expressions.

  What left is to check that the tangent part of $[A_i, B_j]^-$ is tangent to $M$ for $i, j=1,2$ according to Thm.~\ref{hrd} and Eq.~(\ref{Hif}). We will only compute $[A_1, B_1]^-$ and $[A_2, B_1]^-$ and the details are included in the appendix. The computation of $[A_1, B_2]^-$ and $[A_2, B_2]^-$ is similar and left to the interested reader.

  The tangent part of $[A_1, B_1]^-$ is\footnote{Note that to obtain the expression, the equation $\sum_{p=0}^2\partial_{z_p} f_i=0$ is used.}
  \begin{eqnarray*}
  &\quad&(z_1+f_1\bar{z})(z_1-z_0)\sum_{q=0}^2\partial_{z_q}f_0\partial_{\bar{z}_q}+2f_0(z_1+f_1\bar{z})(\partial_{\bar{z}_0}-\partial_{\bar{z}_1})\\
  &+&(z_0+f_0\bar{z})(z_0-z_1)\sum_{q=0}^2\partial_{z_q}f_1\partial_{\bar{z}_q}+2f_1(z_0+f_0\bar{z})(\partial_{\bar{z}_1}-\partial_{\bar{z}_0}).\\
    \end{eqnarray*}

  Its contraction with $d\mu$ is, up a common factor $z_1-z_0$,
  \begin{eqnarray*}
  &\quad&(z_1+f_1\bar{z})\sum_{q=0}^2z_q\partial_{z_q}f_0-2f_0(z_1+f_1\bar{z})\\&-&(z_0+f_0\bar{z})\sum_{q=0}^2z_q\partial_{z_q}f_1
  +2f_1(z_0+f_0\bar{z})\\&=&(z_1+f_1\bar{z})(\sum_{q=0}^2z_q\partial_{z_q}f_0-2f_0)-(z_0+f_0\bar{z})(\sum_{q=0}^2z_q\partial_{z_q}f_1-2f_1)\\
  &=&0,
  \end{eqnarray*}
  where the last equality is due to the fact that $f_0$ and $f_1$ are homogeneous functions of degree 2.

  Similarly, the tangent part of $[A_2, B_1]^-$ is
  \begin{eqnarray*}&\quad&(z_2+f_2\bar{z})(z_1-z_0)\sum_{q=0}^2\partial_{z_q}f_0\partial_{\bar{z}_q}+2f_0(z_2+f_2\bar{z})(\partial_{\bar{z}_0}-\partial_{\bar{z}_1})\\
  &-&(z_0+f_0\bar{z})(z_1-z_0)\sum_{q=0}^2\partial_{z_q}f_2\partial_{\bar{z}_q}+2f_2(z_0+f_0\bar{z})(\partial_{\bar{z}_1}-\partial_{\bar{z}_0}).
 \end{eqnarray*}
 Its contraction with $d\mu$ is, up to a common factor $z_1-z_0$,
 \begin{eqnarray*}
 &\quad&(z_2+f_2\bar{z})\sum_{q=0}^2z_q\partial_{z_q}f_0-2f_0(z_2+f_2\bar{z})\\ &-&(z_0+f_0\bar{z})\sum_{q=0}^2z_q\partial_{z_q}f_2+2f_2(z_0+f_0\bar{z})\\
 &=&(z_2+f_2\bar{z})(\sum_{q=0}^2z_q\partial_{z_q}f_0-2f_0)-(z_0+f_0\bar{z})(\sum_{q=0}^2z_q\partial_{z_q}f_2-2f_2)\\
 &=&0,
 \end{eqnarray*}
 where the last equality is due to the fact that $f_0$ and $f_2$ are homogeneous functions of degree 2.

A similar computation shows the tangent parts of $[A_1, B_2]^-$ and $[A_2, B_2]^-$ are each tangent to $M$. Therefore, we have checked that any associated line bundle of $M$ as a principal $S^1$-bundle is generalized holomorphic.

The associated line bundle of the canonical representation of $S^1$ is actually the tautological line bundle of $\mathbb{C}P^2$, and thus its first Chern class is $-[l]$, where $[l]$ denotes the homology class represented by a line in $\mathbb{C}P^2$. But the first Chern class of the canonical line bundle of $\tilde{\mathbb{J}}_1^\epsilon$ is $-3[l]$. Therefore our construction really gives rise to new generalized holomorphic line bundles.

It is expected that the approach illustrated here could also be applied to construct generalized holomorphic vector bundles of higher rank. We will turn to this elsewhere in the future.
\section{Appendix}
This appendix contains the detailed computation of $[A_1, B_1]^-$ and $[A_2, B_1]^-$.

First a direct computation gives the following formula to be used later:
\[[\bar{E}_i+f_iF, \bar{F}_j+C]=\sum_{q=0}^2\partial_{z_q}f_i(\bar{F}_q+C)-\sum_{q=0}^2\partial_{z_q}f_i(\bar{E}_q+f_qF).\]
Note that the result is independent of $j$.

\begin{eqnarray*}
[A_1, B_1]^-&=&[(z_1+f_1\bar{z})(\bar{E}_0+f_0F), (z_1+h)(\bar{F}_0+C)]^-\\
&-&[(z_1+f_1\bar{z})(\bar{E}_0+f_0F), (z_0+h)(\bar{F}_1+C)]^-\\
&-&[(z_0+f_0\bar{z})(\bar{E}_1+f_1F), (z_1+h)(\bar{F}_0+C)]^-\\
&+&[(z_0+f_0\bar{z})(\bar{E}_1+f_1F), (z_0+h)(\bar{F}_1+C)]^-\\
&=&(z_1+f_1\bar{z})(z_1+h)\Sigma_{q=0}^2\partial_{z_q}f_0(\bar{F}_q+C)\\
&+&2f_0(z_1+f_1\bar{z})(\bar{F}_0+C)-2f_0(z_1+f_1\bar{z})(\bar{F}_1+C)\\
&-&(z_1+f_1\bar{z})(z_0+h)\Sigma_{q=0}^2\partial_{z_q}f_0(\bar{F}_q+C)\\
&-&(z_0+f_0\bar{z})(z_1+h)\Sigma_{q=0}^2\partial_{z_q}f_1(\bar{F}_q+C)\\
&-&2f_1(z_0+f_0\bar{z})(\bar{F}_0+C)+2f_1(z_0+f_0\bar{z})(\bar{F}_1+C)\\
&+&(z_0+f_0\bar{z})(z_0+h)\Sigma_{q=0}^2\partial_{z_q}f_1(\bar{F}_q+C)\\
&=&(z_1+f_1\bar{z})(z_1-z_0)\Sigma_{q=0}^2\partial_{z_q}f_0\bar{F}_q\\
&+&2f_0(z_1+f_1\bar{z})(\bar{F}_0-\bar{F}_1)\\
&+&(z_0+f_0\bar{z})(z_0-z_1)\Sigma_{q=0}^2\partial_{z_q}f_1\bar{F}_q\\
&+&2f_1(z_0+f_0\bar{z})(\bar{F}_1-\bar{F}_0).
\end{eqnarray*}
Similarly,
\begin{eqnarray*}
[A_2, B_1]^-&=&[(z_2+f_2\bar{z})(\bar{E}_0+f_0F), (z_1+h)(\bar{F}_0+C)]^-\\
&-&[(z_2+f_2\bar{z})(\bar{E}_0+f_0F), (z_0+h)(\bar{F}_1+C)]^-\\
&-&[(z_0+f_0\bar{z})(\bar{E}_2+f_2F), (z_1+h)(\bar{F}_0+C)]^-\\
&+&[(z_0+f_0\bar{z})(\bar{E}_2+f_2F), (z_0+h)(\bar{F}_1+C)]^-\\
&=&(z_2+f_2\bar{z})(z_1+h)\Sigma_{q=0}^2\partial_{z_q}f_0(\bar{F}_q+C)\\
&+&2f_0(z_2+f_2\bar{z})(\bar{F}_0+C)-2f_0(z_2+f_2\bar{z})(\bar{F}_1+C)\\
&-&(z_2+f_2\bar{z})(z_0+h)\Sigma_{q=0}^2\partial_{z_q}f_0(\bar{F}_q+C)\\
&-&(z_0+f_0\bar{z})(z_1+h)\Sigma_{q=0}^2\partial_{z_q}f_2(\bar{F}_q+C)\\
&-&2f_2(z_0+f_0\bar{z})(\bar{F}_0+C)+2f_2(z_0+f_0\bar{z})(\bar{F}_1+C)\\
&+&(z_0+f_0\bar{z})(z_0+h)\Sigma_{q=0}^2\partial_{z_q}f_2(\bar{F}_q+C)\\
&=&(z_2+f_2\bar{z})(z_1-z_0)\Sigma_{q=0}^2\partial_{z_q}f_0\bar{F}_q\\
&+&2f_0(z_2+f_2\bar{z})(\bar{F}_0-\bar{F}_1)\\
&+&(z_0+f_0\bar{z})(z_0-z_1)\Sigma_{q=0}^2\partial_{z_q}f_2\bar{F}_q\\
&+&2f_2(z_0+f_0\bar{z})(\bar{F}_1-\bar{F}_0).
\end{eqnarray*}

\section*{Acknowledgemencts}
 This study is supported by the Natural Science Foundation of Jiangsu Province (BK20150797).
\bibliographystyle{elsarticle-num}
\bibliography{<your-bib-database>}

\begin{thebibliography}{00}
 \bibitem{BCG1}
 H. Bursztyn, G. R. Cavalcanti, and M. Gualtieri, Reduction of Courant algebroids and generalized complex structures, Adv. Math. 211, no. 2, 726-765, 2007.
\bibitem{BCG2}
 H. Bursztyn, G. R. Cavalcanti, and M. Gualtieri, Generalized Kaehler geometry of instanton moduli spaces, Commun. Math. Phys. 333, no. 2, pp. 831-860
 \bibitem{Ca}
G. R. Cavalcanti, Reduction of metric structures on Courant algebroid, J. Symplectic Geom. 4, no. 3, 317-343, 2006.
\bibitem{DM}
 R. Dijkgraaf and G. Moore, Balanced topological field theories, Commun. Math. Phys. 185, 411-440, 1997.
 \bibitem{Du}
 T. Drummond, Generalized reduction and pure spinors, J. Symplectic Geom. 12, no. 3, 435-471, 2014.
 \bibitem{Gu00}
M. Gualtieri, Generalized complex geometry, PhD thesis, Oxford University, 2003. arXiv: math./0401221.
 \bibitem{Gu0}
M. Gualtieri, Generalized complex geometry, Ann. of Math, 174: pp. 75-123, 2011.
  \bibitem{Gu1}
 M. Gualtieri, Branes on Poisson varieties, in The many facets of geometry, 368-394. Oxford Univ. Press, Oxford, 2010.
 \bibitem{Hi}
 N. Hitchin, Lectures on generalized geometry, in "Surveys in Differential Geometry Vol. 16", N.- C. Leung and S.-T. Yau, (eds.), International Press, Cambridge, Mass. 79-124, 2011. arXiv:1008.0973.
 \bibitem{Hi2}
 N. Hitchin, Instantons, Poisson structures and generalized Kaehler geometry, Comm. Math. Phys. 265, no.1, 131-164, 2006. arXiv:math/0503432v1.
 \bibitem{Hu}
 S. Hu, R. Morarua, and R. Seyyedali, A Kobayashi-Hitchin correspondence for $I_\pm$-holomorphic bundles, Adv. Math. 287, no. 10, 519-566, 2016.
 \bibitem{LT}
  Y. Lin, and S. Tolman, Symmetries in generalized Kaehler geometry, Commun. Math. Phys. 268, 199-222, 2006. arXiv:math./0509069
  \bibitem{Wang0}
  Y. Wang, generalized holomorphic structures, J. Geom. Phys. 61, 1976-1984, 2011.
  \bibitem{Wang1}
  Y. Wang, generalized holomorphic structures, J. Geom. Phys. 86, 273-283, 2014.
    \bibitem{Wang}
 Y. Wang, Metric reduction in generalized geometry and balanced topological field theories, arXiv:1708.00567.
 \bibitem{Wang2}
 Y. Wang, The GIT aspect of generalized Kaehler reduction. I, arxiv.org/pdf/1803.01178.
\end{thebibliography}



\end{document}